\documentclass[11pt]{article}
\usepackage{blindtext}
\usepackage{enumitem} 
\setlist[enumerate]{parsep=0pt}
 
\usepackage{ragged2e} 

\usepackage{mathrsfs}
\usepackage{amssymb}
\usepackage{layout}
\usepackage{graphicx}
\usepackage{marvosym}
\usepackage[usenames,dvipsnames]{color}
\usepackage{verbatim}
\usepackage{enumitem}
\usepackage[pdftex]{hyperref}
\usepackage{fancyhdr}
\usepackage{indentfirst}
\usepackage{centernot}
\usepackage[all]{xy}

\usepackage{amsthm,etoolbox}

\makeatletter
\@addtoreset{theorem}{section}
\@addtoreset{theorem}{subsection}
\@addtoreset{theorem}{subsubsection}

\usepackage{amsmath}

\usepackage{tikz-cd}
\usetikzlibrary{calc}

\usepackage{hyperref}
\usepackage{cleveref}[noabbrev]
\newtheorem{theorem}{Theorem}

\theoremstyle{definition}

\newtheorem{definition}[theorem]{Definition}

\theoremstyle{theorem}

\newtheorem{lemma}[theorem]{Lemma}
\newtheorem{prop}[theorem]{Proposition}
\newtheorem{cor}[theorem]{Corollary}
\newtheorem{fact}[theorem]{Fact}
\newtheorem{remark}[theorem]{Remark}

\crefname{theorem}{Theorem}{Theorems}
\crefname{lemma}{Lemma}{Lemmas}
\crefname{prop}{Proposition}{Propositions}
\crefname{fact}{Fact}{Facts}
\crefname{remark}{Remark}{Remarks}
\crefname{cor}{Corollary}{Corollaries}

\newcommand{\theoremprefix}{}
\let\thetheoremsaved\thetheorem
\renewcommand{\thetheorem}{\theoremprefix\thetheoremsaved}

\patchcmd{\@startsection}{\par}{\renewcommand{\theoremprefix}{\csname the#1\endcsname.}}{}{}

\makeatother

\begin{document}

\clearpage
\pagenumbering{arabic}

\def\tp{\mbox{\rm tp}}
\def\qftp{\mbox{\rm qftp}}
\def\cb{\mbox{\rm cb}}
\def\wcb{\mbox{\rm wcb}}
\def\Diag{\mbox{\rm Diag}}
\def\trdeg{\mbox{\rm trdeg}}
\def\Gal{\mbox{\rm Gal}}
\def\Lin{\mbox{\rm Lin}}

\def\restriction#1#2{\mathchoice
              {\setbox1\hbox{${\displaystyle #1}_{\scriptstyle #2}$}
              \restrictionaux{#1}{#2}}
              {\setbox1\hbox{${\textstyle #1}_{\scriptstyle #2}$}
              \restrictionaux{#1}{#2}}
              {\setbox1\hbox{${\scriptstyle #1}_{\scriptscriptstyle #2}$}
              \restrictionaux{#1}{#2}}
              {\setbox1\hbox{${\scriptscriptstyle #1}_{\scriptscriptstyle #2}$}
              \restrictionaux{#1}{#2}}}
\def\restrictionaux#1#2{{#1\,\smash{\vrule height .8\ht1 depth .85\dp1}}_{\,#2}} 

\newcommand{\forkindep}[1][]{%
  \mathrel{
    \mathop{
      \vcenter{
        \hbox{\oalign{\noalign{\kern-.3ex}\hfil$\vert$\hfil\cr
              \noalign{\kern-.7ex}
              $\smile$\cr\noalign{\kern-.3ex}}}
      }
    }\displaylimits_{#1}
  }
}
\newpage

\begin{center}

\large \MakeUppercase{A note on weakly stable Kim-forking}

\vspace{5mm}

    \large Yvon \textsc{Bossut}
\end{center}
\vspace{10mm}

Abstract : We define weak stable Kim-forking, a notion that generalizes stable forking to the context of NSOP$_1$ theories. We adapt some of the known results on stable forking to this context.

\section{ Introduction}

NSOP$_1$ theories have recently been studied as a generalization of simple theories. The notion of Kim-forking plays a similar role in NSOP$_1$ theories than forking in simple theories. Kim and Pillay \cite{kim1997simple} have shown that forking independence in simple theories is characterized by some of its properties (see \cite{wagner2000simple} for example). Chernikov and Ramsey have shown similar results for Kim-independence in NSOP$_1$ theories, \cite[Theorem 5.8]{chernikov2016model}.

\vspace{10pt}
Here Kim-forking is considered at the level of formulas. The stable forking conjecture is the conjecture that in a simple theory, if a complete type forks over a set of parameter, then there is a stable formula inside of this type which forks. In short, that forking is witnessed by stable formulas. A generalization of the notion of stable forking to the context of NSOP$_1$ theories, namely weak stable Kim-forking, is given here. We try to connect this notion to the results known about stable forking, which can be found in \cite{casanovas2018stable}. For example it is known that is some theory $T$ has stable forking then the imaginary extension $T^{eq}$ also does. However I could not prove the equivalent result here for some reasons discussed in Section 5.

\vspace{10pt}
We can also cite on this subject \cite{baldwin2024simple} a recent article by Baldwin, Freitag and Mutchnik in which they demonstrate, among other things, that Kim-forking in an NSOP$_1$ theory is always witnessed by a simple formula, that is, one that does not have the tree property.

\section{Preliminaries : Independence relations and local types}

\subsection{Without the existence axiom}

\vspace{10pt}
Let $\mathbb{M}\models T$ be a monster model of some complete theory. We introduce the relations of independence that will appear here.

\begin{definition} Let $b_{0},e$ be tuples of $\mathbb{M}$ and $\varphi(x,b_{0})$ be a formula.\begin{enumerate}
\item $\varphi(x,b_{0})$ \emph{divides over $e$} if there is an $e$-indiscernible sequence $I=(b_{i}$ : $i < \omega)$ such that :

\vspace{-2pt}
\begin{center}
$ \lbrace \varphi (x,b_{i})$ : $i<\omega \rbrace$ is inconsistent.
\end{center}
\vspace{-2pt}

In that case we say that $\varphi(x,b_{0})$ divides over $e$ with respect to $I$.
\item A partial type $p (x,b)$ \emph{forks over $e$} if it implies a finite disjunction of formulas each of which divides over $e$.
\item Let $a \in \mathbb{M}$. We write $a\forkindep^{f}_{e}b$ to denote the assertion that $\tp(a/eb_{0})$ does not fork over $e$ and $a\forkindep^{d}_{e}b_{0}$ to denote the assertion that $\tp(a/eb_{0})$ does not divide over $e$.
\item A sequence of tuples $(a_{i}$ : $i\in I)$ is called \emph{$e$-Morley} if it is $e$-indiscernible and if $a_{i}\forkindep^{f}_{e}a_{<i} $ for every $i \in I $.
\end{enumerate}
\end{definition}

\begin{definition} A set $e\subseteq \mathbb{M}$ is an \emph{extension basis} if $a\forkindep^{f}_{e}e$ for every $a\in \mathbb{M}$. A theory $T$ has \emph{existence} if every set is an extension basis. This is equivalent to saying that for every $e\in \mathbb{M}$ and every $p(x)\in S(e)$ there is an $e$-Morley sequence in $p$.
\end{definition}

\begin{remark} Every model is an extension basis, simple theories satisfy existence, and it was a conjecture that NSOP$_1$ theories does as well before a counterexample was found by Mutchnik in \cite{mutchnik2024mathrm}.
\end{remark}

\begin{definition} Consider a small model $M \prec \mathbb{M}$. A sequence of tuples $(a_{i}$ : $i\in I)$ is called \emph{Coheir Morley over $M$} if it is $M$-indiscernible and if $\tp(a_{i}/Ma_{<i}$ is finitely satisfiable in $M$ for every $i \in I $. Such a sequence is always $M$-Morley.
\end{definition}

\vspace{10pt}
There are a lot of simplicity-like results characterizing NSOP$_1$-theories in terms of properties of Kim-forking, for now let us set the following definition :

\begin{definition} Let $T$ be a complete theory. $T$ is NSOP$_1$ if and only if Kim-independence is symmetric over models, meaning that $a\forkindep^{K}_{M}b$ iff $b\forkindep^{K}_{M}a$ for every $a,b\in \mathbb{M}$ and $M\prec \mathbb{M}$.\end{definition}

\vspace{10pt}
Without the hypothesis of existence we define Kim-forking over models, this notion has been studied in \cite{kaplan2020kim} and \cite{kaplan2019local} for example.

\begin{definition} Consider a small model $M \prec \mathbb{M}$. Let $b_{0}\in \mathbb{M}$ and let $\varphi(x,b_{0})$ be a partial type over $Mb$.\begin{enumerate}

\item $\varphi(x, b_{0})$ \emph{Kim-divides over $M$} if there is an $M$-Coheir-Morley sequence $(b_{i}$ : $i<\omega )$ such that :

\vspace{-2pt}
\begin{center}
$ \lbrace \varphi (x,b_{i})$ : $i<\omega \rbrace$ is inconsistent.
\end{center}
\vspace{-2pt}
We will say that $\varphi(x, b_{0})$ Kim-divides over $M$ with respect to $(b_{i}$ : $i<\omega)$.

\item $p(x,b)$ \emph{Kim-forks over $M$} if it implies a finite disjunction of formulas each of which Kim-divides over $M$.

\item Let $a \in \mathbb{M}$. We write $a\forkindep^{K}_{M}b_{0}$ to denote the assertion that $\tp(a/Mb)$ does not Kim-fork over $M$ and $a\forkindep^{Kd}_{M}b$ to denote the assertion that $\tp(a/Mb)$ does not Kim-divide over $M$.
\end{enumerate}
\end{definition}

\vspace{10pt}
This is the notion of Kim-forking that we will use in Section 3.

\subsection{Assuming existence}

If we assume that the theory we are working in satisfies existence we can define Kim-forking over arbitrary sets :

\begin{definition} Consider a tuple $e \in \mathbb{M}$. Let $b_{0}\in \mathbb{M}$ and $\varphi(x,b_{0})$ be a partial type over $eb$. $\varphi(x, b_{0})$ \emph{Kim-divides over $e$} if there is an $e$-Morley sequence $(b_{i}$ : $i<\omega )$ such that $\lbrace \varphi(x,b_{i})$ : $i<\omega \rbrace$ is inconsistent. We will say that $\varphi(x, b_{0})$ Kim-divides over $e$ with respect to $(b_{i}$ : $i<\omega)$. Kim-forking over $e$ is defined similarly. 
\end{definition}

\vspace{10pt}
This is the notion of Kim-forking that we will use in Section 4, its properties in NSOP$_1$ theories have been studied in \cite{kaplan2021transitivity} and \cite{dobrowolski2022independence} for example. To give some example the following notion will also appear in this section :

\begin{definition} We define \emph{algebraic independence} as the relation $A\forkindep^{a}_{C}B:=  acl(AC)\cap acl(BC)=acl(C)$.\end{definition}

In this section we consider Kim-forking at the level of formulas. The stable forking conjecture is the conjecture that in a simple theory, if a complete type forks over a set of parameter, then there is a stable formula inside of this type which forks. In short, that forking is witnessed by stable formulas.

\vspace{10pt}
A generalization of the notion of stable forking to the context of NSOP$_1$ theories, namely weak stable Kim-forking, is given here. We connect this notion to some results known about stable forking which can be found in \cite{casanovas2018stable}. It is known that if a theory $T$ has stable forking then the imaginary extension $T^{eq}$ also does. However I could not prove the equivalent result here for some reasons discussed in Section 5.

\subsection{Local types}

We recall some basic notions around local types. Let $\varphi(x,y)$ be a formula.

\begin{definition} A \emph{$\varphi$-formula over $A$} is a boolean combination of formulas of the form $\varphi(x,a)$ for $a\in A$. A \emph{generalized $\varphi$-formula over $A$} is a formula over $A$ that is equivalent to a boolean combination of $\varphi$-formulas, possibly with parameters outside of $A$. 

\vspace{10pt}
A \emph{complete $\varphi$-type over $A$} (resp. \emph{complete generalized $\varphi$-type over $A$}) $p$ is a consistent set of $\varphi$-formulas (resp. generalized $\varphi$-formulas) over $A$ such that given any $\psi(x)$ $\varphi$-formula (resp. generalized $\varphi$-formula) over $A$ either $\psi(x)\in p$ or $\neg\psi(x) \in p$.

\vspace{10pt}
The set of complete $\varphi$-types (resp. generalized $\varphi$-types) over $A$ is written $S_{\varphi}(A)$ (resp. $S_{\varphi^{*}}(A)$) and the $\varphi$-type (resp. generalized $\varphi$-type) of $a$ over $A$ is written $\tp_{\varphi}(a/A)$ (resp. $\tp_{\varphi^{*}}(a/A)$).
\end{definition}

\begin{definition} A $\varphi$-type $p$ is said to be definable over some set of parameters $A$ if there is a formula $d_{p}x\varphi(x,y)$ with parameters in $A$ such that $\varphi(x,b)\in p$ if and only if $\models d_{p}x\varphi(x,b)$. 
\end{definition}

\begin{fact}\cite[Chapter 2]{kim1997simple}\label{factmodstable} If $\varphi(x,y,e)$ is a stable formula and $M$ is a model such that $e\in M$, then every $\varphi(x,y,e)$-type $p$ over $M$ is definable over $M$ and has a unique non-forking extension, which has the same definition. We write $d_{p}x\varphi(x,y,e)$ its definition and $Cb(d_{p}x\varphi(x,y,e))$ the canonical parameter of this definition.
\end{fact}

The previous result holds more generally for types over algebraically closed sets in $T^{eq}$.

\begin{fact}\label{defalgclos} \cite[Proposition 6.13]{casanovas2011simple}: Let $\varphi(x,y)$ be a stable formula and $E=acl^{eq}(E) \subseteq B$. Then any generalized $\varphi$-type $p_{\varphi}\in S_{\varphi^{*}}(E)$ is definable over $E$, and it has a unique extension $q\in S_{\varphi^{*}}(B)$ which is also definable over $E$. The formula $d_{p}x\varphi(x,y)$ defining $p$ then admits a canonical parameter in $E$, which we write $Cb(d_{p}x\varphi(x,y))$.
\end{fact}

\section{Weak stable Kim-forking over models}

In this section we work in any complete theory $T$ in a language $\mathcal{L}$, all sets of parameters and elements are assumed to be in a monster model $\mathbb{M}$ of $T$. We show some basic results about what can and can not happen when Kim-dividing is witnessed by stable formulas, and isolate the notion of 'weak stable Kim-forking' as the correct generalization of stable forking in simple theories to the NSOP$_1$ context. In this section we only consider types over models and Kim-forking over models.

\begin{definition} A theory $T$ has \emph{stable Kim-forking over models} (resp. \emph{stable forking over models}) if whenever $a\centernot\forkindep^{K}_{M}B$ for $M\subseteq B$ and $M$ a model (resp. $a\centernot\forkindep^{f}_{M}B$) there is a stable formula $\varphi(x,y) \in \mathcal{L}$ and $b\in B$ such that $a\models \varphi(x,b)$ and $\varphi(x,b)$ Kim-forks (resp. forks) over $M$.
\end{definition}

\begin{definition} A theory $T$ has \emph{weak stable Kim-forking over models} (resp. \emph{weak stable forking over models}) if whenever $a\centernot\forkindep^{K}_{M}B$ (resp. $a\centernot\forkindep^{f}_{M}B$) for $M\subseteq B$ and $M$ a model there is a stable formula $\varphi(x,y,m) \in \mathcal{L}_{m}$ (i.e. the formula with fixed $m\in M$ is stable) and $b\in B$ such that $\models \varphi(a,b,m)$ and $\varphi(x,b,m)$ Kim-forks (resp. forks) over $M$.
\end{definition}

It is known that a theory with stable forking is simple \cite[Proposition 1.3]{casanovas2018stable}. We show that a theory with stable Kim-forking over models is simple and that a theory with weak stable forking over models is simple. We also show that a theory with weak stable Kim-forking over models is NSOP$_1$ and not necessarily simple.

\begin{lemma}\label{fstable} Let $M\subseteq N$ be models of a theory $T$ with weak stable forking over models. For any type $p=\tp(a/N)$,  $a\forkindep^{f}_{M}N$ if and only if for every stable formula $\varphi(x,y,m)$ with $m\in M$ the canonical parameter $Cb(d_{p}x\varphi(x,y,m))$ is in $M$.\end{lemma}

\begin{proof}
$[\implies]$ By \cref{factmodstable} we know that the $\varphi(x,y,m)$-type of $a$ over $M$ is definable over $M$ and admits a unique non-forking extension to $N$, which has the same definition over $M$, so in particular the same canonical parameter $Cb(d_{p}x\varphi(x,y,m))\in M$.

\vspace{10pt}
$[\impliedby]$ Conversely assume that $p$ forks over $M$ and that $M$ contains the canonical parameters. Then there is a stable formula $\varphi(x,y,m)$ with $m\in M$ and $n_{0}\in N$ such that $\varphi(x,n_{0},m) \in p$ and $\varphi(x,n_{0},m)$ forks over $M$.

\vspace{10pt}
So there is an $M$-indiscernible sequence $(n_{i})_{i<\omega}$ such that $\lbrace \varphi(x,n_{i},m)$ : $i<\omega \rbrace $ is inconsistent. Now since $Cb(d_{p}x\varphi(x,y,m))\in acl^{eq}(M)$ by indiscernibility $\models d_{p}x\varphi(x,n_{i},m)$ for all $i<\omega$, so $\lbrace \varphi(x,n_{i},m)$ : $i<\omega \rbrace$ is satisfied by any realization of a non-forking extension of $p$ to $N(n_{i})_{i<\omega}$, a contradiction.
\end{proof}

\begin{prop}\label{simpleici} Let $T$ be a theory with weak stable forking over models, then $T$ is simple.
\end{prop}

\begin{proof}
We show that $\forkindep^{f}$ satisfies local character. Let $a$ be a finite tuple, $M$ a model, $p=\tp(a/M)$. We consider the set $C_{0}=\lbrace Cb(d_{p}x\varphi(x,y))$: $\varphi(x,y)\in \mathcal{L}$ a stable formula$\rbrace$. We take a model $M_{0}\subseteq M$ such that $\vert M_{0} \vert = \vert T \vert$ and $M_{0}$ contains a representative of every element of $C_{0}$.

\vspace{10pt}
We iterate the process taking $C_{i+1}=\lbrace Cb(d_{p}x\varphi(x,y,e))$ : $\varphi(x,y,e)\in \mathcal{L}_{e}$ a stable formula with $e\in M_{i}\rbrace$ and $M_{i}\subseteq M_{i+1}\subseteq M$ such that $\vert M_{i+1} \vert = \vert T \vert$ and $M_{i+1}$ contains a representative of every element of $C_{i+1}$. Let $M':=\bigcup_{i<\omega} M_{i}$. By \cref{fstable} $a\forkindep^{f}_{M'}M$ and $M'$ has size $\vert T \vert$.
\end{proof}

\begin{lemma}\label{kfstable} Let $\mathbb{M}$ be a monster model of a theory $T$ with weak stable Kim-forking over models. Let $M\subseteq B \subseteq \mathbb{M}$ with $M$ a model and $p:=\tp(a/B)$ a type. Then $a\forkindep^{K}_{M}B$ if for every stable formula $\varphi(x,y,m)$ with $m\in M$ the canonical parameter $Cb(d_{p}x\varphi(x,y,m))$ is in $M$.\end{lemma}

\begin{proof}
Assume that $p$ Kim-forks over $M$ and that $M$ contains all of the canonical parameters. There is a stable formula $\varphi(x,y,m)$ with $m\in M$ and $n_{0}\in N$ such that $\varphi(x,n_{0},m) \in p$ and $\varphi(x,n_{0},m)$ Kim-forks over $M$. So there is an $M$-Morley sequence $(n_{i})_{i<\omega}$ such that $\lbrace \varphi(x,n_{i},m)$: $i<\omega \rbrace $ is inconsistent.

\vspace{10pt}
By assumption $Cb(d_{p}x\varphi(x,y,m))\in M$, so by indiscernibility $\models d_{p}x\varphi(x,n_{i},m)$ for all $i<\omega$, and $\lbrace \varphi(x,n_{i},m)$ : $i<\omega \rbrace$ is satisfied by any realization of a non-forking extension of $p$ to $N(n_{i})_{i<\omega}$, which contradicts the fact that $\varphi(x,n_{0},m)$ Kim-forks over $M$.\end{proof}

\begin{prop}\label{nsop1stableweak} Let $T$ be a theory with weak stable Kim-forking over models, then $T$ is NSOP$_1$.
\end{prop}

\begin{proof}
We show that $\forkindep^{K}$ satisfies the local character criterion of \cite[Proposition 2.6]{chernikov2023transitivitylownessranksnsop1} to show that $T$ is NSOP$_1$. Assume that $T$ has weak stable Kim-forking and that there exists a continuous increasing sequence of models of size $\vert T \vert$ $(M_{i})_{i<\kappa}$ for $\kappa = \vert T \vert^{+}$ and a finite tuple $a$ such that $a\centernot\forkindep^{K}_{M_{i}}M_{i+1}$ for all $i<\kappa$. Let $M:=\bigcup_{i<\kappa} M_{i}$ and $p:= \tp(a/M)$.

\vspace{10pt}
For every $i<\kappa$ the number of stable formulas with parameters in $M_{i}$ is at most $\vert T \vert$, since $\kappa > \vert T \vert$ is regular there is a smallest index $\alpha_{i}\geq i$ such that for every stable formula $\varphi(x,y,e)$ with $e\in M_{i}$, $Cb(d_{p}x\varphi(x,y,e))\in M_{\alpha_{i}}$. Now let us consider $(\beta_{i})_{i<\omega}$ defined by induction by $\beta_{0}=0$ and $\beta_{i+1}=\alpha_{\beta_{i}}$. Take $\beta = \bigcup_{i<\omega} \beta_{i}$, by continuity $\alpha_{\beta}=\beta$, so $a\forkindep^{K}_{M_{\beta}}M$ by \cref{kfstable}, which contradicts $a\centernot\forkindep^{K}_{M_{\beta}}M_{\beta+1}$.\end{proof}

\begin{cor} Let $T$ be a theory with stable Kim-forking over models, then $T$ is simple.
\end{cor}

\begin{proof}
By \cref{nsop1stableweak} we know that $T$ is NSOP$_1$. If $T$ is non-simple we can build an arbitrary long sequence of Kim-forking extensions with an instance of TP$_2$ (as in \cite[Proposition 3.6]{casanovas2019more}). Now we can run the same proof as in \cite[Proposition 1.3]{casanovas2018stable}: There is a stable formula with a dividing chain of infinite length and hence with the tree property, which is not possible.\end{proof}

\begin{lemma}\label{uniqueextkimf} Let $T$ be an  theory. If $p=\tp(a/M)$ is a complete type over a model $M$, $M\subseteq N$ is a model, and $\varphi(x,y,m)$ is a stable formula for some fixed $m\in M$, then there is a unique complete $\varphi$-type extending $p$ that does not Kim-fork over $M$. The unique global $\varphi(x,y,e)$-type $q_{\varphi}$ extending $p$ without Kim-forking is definable over $M$, we write $d^{K}_{p}x\varphi(x,y,e)$ its definition and $Cb(d^{K}_{p}x\varphi(x,y,e))$ the canonical parameter of this definition.
\end{lemma}

\begin{proof} Assume that there are $a',a''\models p$ such that $a'\forkindep^{K}_{M}N$, $a''\forkindep^{K}_{M}N$ and that for some $b\in N$, $\models \varphi(a',b,m)$ and $\models \neg \varphi(a'',b,m)$. Consider an $M$-Coheir Morley sequence $(N_{i})_{i<\omega}$ such that $N=N_{0}$. Let $(b_{i})_{i<\omega}$ be the corresponding sequence with $b_{0}=b$, and let $a'_{i},a''_{i}$ be such that $a'_{i}N_{i}\equiv_{M}a'N_{0}$ and $a''_{i}N_{i}\equiv_{M}a''N_{0}$, so $\models \varphi(a'_{i},b_{i},m)$ and $\models \neg \varphi(a''_{i},b_{i},m)$ for every $i<\omega$.

\vspace{10pt}
By applying the amalgamation theorem for Kim-forking over $M$ along the sequence $N_{i}$ we can construct a sequence $(a_{i})_{i<\omega}$ such that $\models \varphi(a_{i},b_{j},m)$ iff $i<j$, contradicting the stability of $\varphi(x,y,m)$. So $\varphi (x,b,e)\in q_{\varphi}$ iff $p\cup \lbrace \varphi(x,b,e) \rbrace$ does not Kim-fork over $M$, and the complementary of this set is $\lbrace b $ : $p\cup \lbrace \varphi(x,b,e) \rbrace$ Kim forks over $M \rbrace$. Both of these sets are $M$-type definable, so they are definable over $M$.
\end{proof}

\begin{remark} By \cref{factmodstable} we know that there is a unique non-forking extension of $p\in S_{\varphi}(M)$ to any set of parameters, and by the above result we know that there is a unique non Kim-forking extension. So these two extension coincide, and we get that the Shelah definition $d_{p}x\varphi(x,y,e)$ coincides with the definition $Cb(d^{K}_{p}x\varphi(x,y,e))$.
\end{remark}

\begin{lemma}\label{kfstablegen} Let $M\subseteq N$ be models in a theory $T$ with weak stable Kim-forking over models. Let $q=\tp(a/N)$ and $p=\tp(a/M)$. Then $a\forkindep^{K}_{M}N$ if and only if for every stable formula $\varphi(x,y,m)$ with $m\in M$ the canonical parameter $Cb(d^{K}_{q}x\varphi(x,y,m)) \in N$ is equal to $Cb(d^{K}_{q}x\varphi(x,y,m)) \in M$. These two canonical parameters are equal if $Cb(d^{K}_{q}x\varphi(x,y,m))\in M$.\end{lemma}

\begin{proof}
$[\implies]$ Assume that $a\forkindep^{K}_{M}N$ and that $Cb(d_{p}x\varphi(x,y,e))\centernot\in M$ for some stable formula $\varphi(x,y,m)$ with $m\in M$. Let $p_{\varphi}:=\tp_{\varphi}(a/M)$. \cref{uniqueextkimf}, which we can apply by \cref{nsop1stableweak}, yields that there is a unique complete $\varphi(x,y,m)$-type $q_{\varphi}$ over $N$ such that $p\cup q_{\varphi}$ does not Kim-fork over $M$. In particular since a non-forking extension always exist we get that $p\cup q_{\varphi}$ does not fork over $M$, so $q_{\varphi}$ coincides with the unique non-forking extension of $p_{\varphi}$ to $N$, so $q_{\varphi}$ and $p_{\varphi}$ have the same definition and the same canonical parameter.

\vspace{10pt}
$[\impliedby]$ Is just \cref{kfstable}.
\end{proof}

\section{Over algebraically closed sets}

In this section we work in a monster model $\mathbb{M}$ of a theory $T$ with existence. We can define weak stable Kim-forking over algebraically closed sets in a similar fashion. We will make a small use of hyperimaginaries. As previously we will consider stable formulas $\varphi(x,y,e)$ for some fixed parameters, for convenience we will sometimes shorten to $\varphi$ and write $\tp_{\varphi}(a/A)$ instead of $\tp_{\varphi(x,y,e)}(a/A)$ for example.

\begin{definition} A theory $T$ has \emph{weak stable Kim-forking} if whenever $a\centernot\forkindep^{K}_{E}B$ for $E=acl(E)\subseteq B$ there is a stable formula $\varphi(x,y,e) \in \mathcal{L}_{e}$ (i.e. the formula with fixed $e\in E$ is stable) and $b\in B$ such that $a\models \varphi(x,b,e)$ and $\varphi(x,b,e)$ Kim-forks over $E$.
\end{definition}

\begin{remark}\label{factdef} By \cref{defalgclos} we know that if $\varphi(x,y)$ is a stable formula and $E=acl^{eq}(E) \subseteq B$ then any generalized $\varphi$-type $p_{\varphi}\in S_{\varphi^{*}}(E)$ has a unique extension $q\in S_{\varphi^{*}}(B)$ which is also definable over $E$. It is easy to check that this extension does not Kim-fork over $E$:

\vspace{10pt}
Write $p$ as $p(x,B)$. If $(B_{i})_{i<\omega}$ is an $E$-Morley sequence then by indiscernibility the unique extension $q\in S_{\varphi^{*}}((B_{i})_{i<\omega})$ which is also definable over $E$ implies $\cup_{i<\omega}p(x,B_{i})$, so this type is consistent.
\end{remark}

\begin{remark}\label{formstab} \cite[Remark 1.3]{kim2001around} If $\varphi(x,a)\equiv \psi(x,b)$ and $\psi(x,z)$ is stable, then for some $\mu(y)\in \tp(a)$, $\varphi(x,y)\wedge \mu(y)$ is stable. In particular if $\psi(x,b)$ is a generalized $\varphi$-formula for some stable formula $\varphi(x,y)$, $\psi(x,y)$ might not be stable, but adding a fragment of $\tp(b)$ will give a stable formula.
\end{remark}

\begin{definition} Let $E\subseteq \mathbb{M}$ be a tuple and $p(x)$ be a type over $E$. We call $p$ a \emph{Kim-amalgamation basis} if whenever $p_{1}$ and $p_{2}$ are non-Kim-forking extensions of $p$ over some sets of parameters $A$ and $B$ respectively, with $A,B\subseteq \mathbb{M}^{eq}$, $E\subseteq A\cap B$ and $A\forkindep^{K}_{E}B$, then the type $p_{1}\cup p_{2}$ does not Kim-fork over $E$. Any complete type over a model or a boundedly closed set is a Kim-amalgamation base: \cite[Theorem 2.5.8]{bossut2022kimforking}.\end{definition}

The proof of the following result works similarly as the one of \cref{uniqueextkimf}. Here we need $p$ to be a Kim-amalgamation basis in order to apply the independence theorem, and \cref{formstab} in order to have stable formulas.

\begin{lemma}\label{uniqueexthyp} Assume that $T$ is NSOP$_1$. If $p=\tp(a/E)$ is a Kim-amalgamation basis, $E\subseteq B$, and $\varphi(x,y,e)$ is a stable formula for some fixed $e\in E$, then there is a unique complete generalized $\varphi$-type $p_{\varphi^{*}}$ extending $p$ that does not Kim-fork over $E$ and it is definable over $E$. In particular, for every generalized $\varphi$-formula $\psi(x,b)$ with $b\in B$, exactly one of $p\cup \lbrace \psi (x,b)\rbrace$ and $p\cup \lbrace \neg \psi (x,b)\rbrace$ Kim-forks over $E$.
\end{lemma}

\begin{proof} Assume that there are $a',a''\models p$ such that $a'\forkindep^{K}_{E}B$, $a''\forkindep^{K}_{E}B$ and that for some generalized $\varphi$-formula $\psi(x,b)$ with $b\in B$, $\models \psi(a',b)$ and $\models \neg \psi(a'',b)$. By \cref{formstab} there is a formula $\mu(y)\in \tp(b/E)$ such that $\psi(x,y)\wedge \mu(y)$ is stable. Consider an $E$-Morley sequence $(b_{i})_{i<\omega}$ such that $b=b_{0}$. 

\vspace{10pt}
Let $a'_{i},a''_{i}$ be such that $a'_{i}b_{i}\equiv_{E}a'b_{0}$ and $a''_{i}b_{i}\equiv_{E}a''b_{0}$. Then $\models \psi(a'_{i},b_{i})\wedge \mu(b_{i})$ and $\models \neg \psi(a''_{i},b_{i})\wedge \mu(b_{i})$ for every $i<\omega$. By applying the amalgamation theorem for Kim-forking over $E$ along the sequence $b_{i}$ we can construct a sequence $(a_{i})_{i<\omega}$ such that $\models \psi(a_{i},b_{j})\wedge \mu(b_{j})$ iff $i<j$, contradicting the stability of $\psi(x,y)\wedge \mu(y)$.

\vspace{10pt}
Consider some generalized $\varphi$-formula $\psi(x,b)$ with $b\in B$. $\psi (x,b)\in p_{\varphi^{*}}$ if and only if $p\cup \lbrace \psi(x,b) \rbrace$ does not Kim-fork over $E$, and by completeness $\psi (x,b)\centernot\in p_{\varphi^{*}}$ if and only if $p\cup \lbrace \neg\psi(x,b) \rbrace$ does not Kim-fork over $E$. Both of these sets are $E$-type definable, so they are definable over $E$.
\end{proof}

Using both results of uniqueness we get that \cref{uniqueexthyp} holds more generally for generalized $\varphi$-types over algebraically closed sets in $T^{eq}$.

\begin{cor}\label{uniqueextdef} Assume that $T$ is NSOP$_1$. If $p=\tp_{\varphi}^{*}(a/E)$ with $E=acl^{eq}(E)$, $E\subseteq B$, and $\varphi(x,y,e)$ is a stable formula for some fixed $e\in E$, then there is a unique complete generalized $\varphi$-type extending $p$ that does not Kim-fork over $E$, and such a type does not Kim-fork over $E$ iff it is definable over $E$.
\end{cor}

\begin{proof} Fix some completion $p'\in S(bdd(E))$, by \cref{uniqueexthyp} there is a unique $q'\in S_{\varphi^{*}}(B)$ such that $p'\cup q'$ does not Kim-fork over $E$ (we recall that a type Kim-forks over $E$ if and only if it Kim-forks over $bdd(E)$). So it coincides with the unique extension of $p$ over $B$ that is definable over $E$ from \cref{factdef}.
\end{proof}
 
\begin{remark}\label{translocal} \cref{uniqueextdef} implies that in any NSOP$_1$ theory Kim-forking is transitive for complete generalized $\varphi$-type over algebraically closed sets in the imaginaries: Let $E=acl^{eq}(E)\subseteq B=acl^{eq}(B) \subseteq D$, $p=\tp_{\varphi^{*}}(a/E)$, $p'=\tp_{\varphi^{*}}(a/B)$ and $p''=\tp_{\varphi^{*}}(a/D)$. If $p''$ does not Kim-fork over $B$ and $p'$ does not Kim-fork over $E$ then $p''$ does not Kim-fork over $E$.
\end{remark}

\begin{definition}
\label{kimfstabhyp2} If $p=\tp_{\varphi^{*}}(a/E)$ for $E=acl^{eq}(E)$, $E\subseteq B$, and $\varphi(x,y,e)$ is a stable formula for some fixed $e\in E$, then the unique global $\varphi(x,y,e)$-type $q_{\varphi}$ extending $p_{\varphi}:=\tp_{\varphi_{e}}(a/E)$ without forking over $E$ is definable over $E$, we write $d^{K}_{p}x\varphi(x,y,e)$ its definition and $Cb(d^{K}_{p}x\varphi(x,y,e))$ the canonical parameter of this definition.
\end{definition}

\begin{remark}
By uniqueness we get that when both the Shelah definition $d_{p}x\varphi(x,y,e)$ and the previous one exist they coincide. By considering a non-Kim-forking extension of $\tp(a/E)$ to a model $M$ we get that the definition $d_{p}^{K}x\psi(x,y,e)$, which is a formula with parameters in $acl^{eq}(E)$, is equivalent to a positive boolean combination of formulas of the form $\psi(a',y)$ with $a'\in M$.\end{remark}

\begin{cor}\label{stkfhyp} Assume that $T$ has weak stable Kim-forking. If $p=\tp(a/B)$ with $B=acl^{eq}(B)$ and $E=acl(E)\subseteq B$, then $a\forkindep^{K}_{E}B$ if and only if $Cb(d_{p}x\varphi(x,y,e))\in acl^{eq}(E)$ for every stable formula $\varphi(x,y,e)$ with $e\in E$.
\end{cor}

\begin{proof}$[\impliedby]$ Follows from a similar proof as the one of \cref{kfstable}.

\vspace{10pt}
$[\implies]$ Let $p':=\tp(a/acl^{eq}(E))$. Assume that $a\forkindep^{K}_{E}B$ and consider some stable formula $\varphi(x,y,e)$ with $e\in E$. By \cref{uniqueextdef} we know that there is a unique complete $\varphi(x,y,e)$-type $q_{\varphi}$ over $B$ such that $p'\cup q_{\varphi}$ does not Kim-fork over $E$, so this type coincides with the $\varphi(x,y,e)$-type of $a$ over $B$ since $a\forkindep^{K}_{E}B$, so they have the same definition and the same canonical parameter, which lies in $acl^{eq}(E)$.\end{proof}

\begin{remark}\cite[Remark 1.2]{casanovas2018stable}\label{stablealg} Let $\varphi(x,y,e)$ be a stable formula (with some parameters $e$). If $\theta(v,x,e)\vdash \exists^{=n}x\theta(v,x,e)$ then $\psi (v,y,e)=\exists x(\theta(v,x,e) \wedge\varphi(x,y,e))$ is also stable.
\end{remark}

We will use this fact to get a stable formula from another one by replacing the variables by some 'larger variables' in the sense of the algebraic closure.

\begin{remark}
\cref{stkfhyp} implies that if $T$ is an  theories with existence and weak stable Kim-forking, given some tuples $E=acl^{eq}(E)\subseteq B=acl^{eq}(B)$ and a tuple $a$ there is a smaller subset $E\subseteq B_{0}=acl^{eq}(B_{0})\subseteq B$ such that $a\forkindep^{K}_{B_{0}}B$, in the sense that any $E\subseteq B'=acl^{eq}(B')\subseteq B$ such that $a\forkindep^{K}_{B'}B$ must contain $B_{0}$.  So $T^{eq}$ has weak canonical bases in the sense of \cite[Definition 4.1]{kim2021weak}.\end{remark}

\begin{lemma}\label{trans} Assume that $T$ is NSOP$_1$ and has existence. Let $E\subseteq B \subseteq D$ be sets of parameters, $a$ be a tuple such that $a\forkindep^{K}_{B}D$, $\varphi(x,y,e)$ be a stable formula with $e\in E$ and $d\in D$ be such that $\varphi(x,d,e)$ Kim-forks over $E$ and $\models \varphi(a,d,e)$. Then there is a stable formula $\psi(x,y,e)$ and some $b\in B$ such that $\psi(x,b,e)$ Kim-forks over $E$ and $\models \psi(a,b,e)$.\end{lemma}

\begin{proof}  Let $p=\tp_{\varphi^{*}}(a/acl^{eq}(E))$, $p':=\tp_{\varphi^{*}}(a/acl^{eq}(B))$ and $p''=\tp_{\varphi^{*}}(a/D)$.

\vspace{10pt}
Then $Cb(d^{K}_{p''}x\psi(x,y,e))\in acl^{eq}(B) \setminus acl^{eq}(E)$. In fact since $\varphi(x,d,e)$ Kim-forks over $E$ the definition of a $\varphi(x,y,e)$-type containing $\varphi(x,d,e)$ cannot lie in $acl^{eq}(E)$. On the other hand $\tp_{\varphi^{*}}(a/D)$ does not Kim-forks over $B$, so by uniqueness and \cref{kimfstabhyp2} we get that the definition of $\tp_{\varphi^{*}}(a/B)$ and the one of $\tp_{\varphi^{*}}(a/D)$ coincide, so it lies in $acl^{eq}(B)$.

\vspace{10pt}
$p'$ Kim-forks over $E$ by \cref{translocal}, so there is a generalized $\varphi$-formula $\psi(x,b')$ with $b'\in acl^{eq}(B)$ such that $\models \psi(a,b')$ and $\psi(a,b')$ Kim-forks over $E$. By \cref{formstab} we can assume that $\psi(x,y)$ is stable. Take $b\in B$ such that $b'$ is algebraic over $b$, $\theta(z,y,e)$ a formula witnessing this. By considering a strengthening of $\theta$ we can assume that whenever $\models \theta(\beta',b)$  then $\varphi(x,\beta')$ Kim-divides over $E$. By \cref{stablealg} the formula $\exists y(\theta(z,y,e) \wedge \psi (x,y,e) )$ is stable, is satisfied by $a,b$ and $\exists y(\theta(b,y,e) \wedge \psi (x,y,e) )$ Kim-forks over $E$.\end{proof}

It is enough to show weak stable Kim forking for types over models:

\begin{lemma}\label{modkfstable} Let $T$ be a theory with existence such that whenever $a\centernot\forkindep^{K}_{E}N$ for $N$ a model and $E=acl(E)\subseteq N$ there is $e\in E$, $n\in N$, $\varphi(x,y,e)$ a stable formula such that $\models \varphi(a,n,e)$ and $\varphi(x,n,e)$ Kim-forks over $E$. Then $T$ has weak stable Kim-forking.
\end{lemma}

\begin{proof}Let $a$, $E\subseteq B$ such that $a\centernot\forkindep^{K}_{E}B$. Let $N$ be a model containing $B$ such that $a\forkindep^{f}_{B}N$. So $a\centernot\forkindep^{K}_{E}N$ and by assumption there is a stable formula $\psi(x,y,e)$ such that $\psi(x,n,e)$ Kim-forks over $E$ for some $n\in N$. By \cref{trans} there is a stable formula $\varphi(x,y,e)$ and some $b\in B$ such that $\varphi(x,b,e)$ Kim-forks over $E$ and $\models \varphi(x,y,e)$. The fact that $T$ is , which we need to use \cref{trans}, is a consequence of the fact that the hypothesis of this lemma implies weak stable Kim-forking over models, which implies  by \cref{nsop1stableweak}.\end{proof}

\paragraph{Examples of theories with weak stable Kim-forking:}\hfill \break

\textbf{1 -} The first example we give here is the theory $T_{eq,P}$ of parameterized equivalence relations. This  theory has two sorts $O$ for objects and $P$ for parameters. The language is $\mathcal{L}= \lbrace =, \equiv \rbrace$, with $\equiv$ $\subseteq P\times O^{2}$. We write $x\equiv_{p}y$ to mean that $(p,x,y)\in \equiv$. $T_{eq,P}$ expresses that for any parameter $p\in P$ the relation $\equiv_{p}$ is an equivalence relation on $O$ with an infinity of classes all of which are infinite, and that for different parameters these relations interact randomly (see \cite[Section 6.3]{chernikov2016model} and \cite{bossut2023note} for more details on this theory).

\vspace{10pt}
$T_{eq,P}$ has existence, and for arbitrary sets $E\subseteq A,B$:
\begin{enumerate}
    \item $A\forkindep^{K}_{E}B$ iff $A\cap B = E$ and $O(A)\forkindep^{p}_{O(E)}O(B)$ for every $p\in P(E)$ where $\forkindep^{p}$ is forking independence in the theory $T_{\infty}$ of an equivalence relation with an infinite number of classes all infinite.
    \item $A\forkindep^{f}_{E}B$ iff $A\cap B = E$ and $O(A)\forkindep^{p}_{O(E)}O(B)$ for every $p\in P(B)$.
\end{enumerate}

The first condition is witnessed by stable formulas (namely equalities). For the second one, the difference between Kim-forking and forking here is that in the first case the parameter is inside of the basis.

\vspace{10pt}
For a given parameter $p$ the formula $\varphi(x,y) = (x\equiv_{p}y)$ is stable since it defines an equivalence relation, so Kim-forking is always witnessed by stable formulas in this theory. However if we do not fix $p$, the formula $\psi(x,yz):=(x\equiv_{z}y)$ is unstable, so forking in not always witnessed by stable formulas in this theory (which we already knew from \cref{simpleici}).

\vspace{10pt}
\textbf{2 -} The theory of vector spaces of infinite dimension over a field with weak stable Kim-forking (so in particular over a stable field) with a generic bilinear form $sT^{K}_{\infty}$ (which are studied in the Chapter 3) is an other example. A difference with the previous example is that in this case there is a non-trivial algebraic closure, and that Kim-independence is characterized only for algebraically closed structures: If $E=acl(E)\subseteq A= acl(A), B= acl(B)$ then $A\forkindep^{K}_{E}B$ if and only if $\langle A\rangle \cap \langle B \rangle = \langle E \rangle$ and $K(A)\forkindep^{K}_{K(E)}K(B)$.

\vspace{10pt}
The failure of one of these two conditions is witnessed by some stable formula, but for $sT^{K}_{\infty}$ to have weak stable Kim-forking we need this witness to exist even when $A$ and $B$ are not algebraically closed. This is what we prove now.

\vspace{10pt}
For this we use \cref{stablealg} and the fact that in an NSOP$_1$ theory with existence $A\forkindep^{K}_{E}B$ if and only if $acl(EA)\forkindep^{K}_{acl(E)}acl(EB)$ for any sets $E,A,B$.

\vspace{10pt}
Assume that $a\centernot\forkindep^{K}_{acl(E)}b$, then $acl(Ea)\centernot\forkindep^{K}_{acl(E)}acl(Eb)$, so there is a stable formula $\varphi(x,y,e)$ with $e\in acl(E)$ and $\overline{a}\in acl(Ea)$, $\overline{b}\in acl(Eb)$ such that $\models \varphi(\overline{a},\overline{b},e)$ and $\varphi(x,\overline{b},e)$ Kim-divides over $E$. So we can assume that $\overline{a}\in acl(ea)$ and $\overline{b}\in acl(eb)$. 

\vspace{10pt}
Let $\theta_{a}(\overline{x},a,e)$ be a formula witnessing the fact that $\overline{a}\in acl(Ea)$. Let $\theta_{b}(\overline{y},b,e)$ be a formula witnessing the fact that $\overline{b}\in acl(Eb)$, by strengthening $\theta_{b}$ we can assume that $\varphi(x,\overline{b}',e)$ Kim-divides over $E$ whenever $\models \theta_{b}(\overline{b}',b,e)$. For this we eventually have to enlarge the tuple $e$.

\vspace{10pt}
By $\cref{stablealg}$ the formula $\psi(x,y,e):= \theta_{a}(\overline{x},x,e)\wedge \theta_{b}(\overline{y},y,e) \wedge\varphi(\overline{x},\overline{y},e)$ is stable, by construction $\models \psi(a,b,e)$ and the formula $\psi(x,b,e)$ Kim-forks over $E$.

\vspace{10pt}
\textbf{3 -} With a similar proof we can show that if an NSOP$_1$ theory satisfies that $\forkindep^{K}=\forkindep^{a}$ over arbitrary sets then it has weakly stable Kim-forking. In fact in that case $a\centernot\forkindep^{K}_{e}b$ is witnessed by $\exists x\exists y (\varphi(x,a,e)\wedge \psi(y,b,e) \wedge x=y)$ with $\varphi(x,a,e)$ and $\psi(y,b,e)$ algebraic formulas, and the stable formula we apply \cref{stablealg} to is simply $(x=y)$. Examples of such a theories are given in the Chapter 2 of this manuscript.

\section{The question of imaginaries}

It is shown in section $2$ of \cite{casanovas2018stable} that if a theory $T$ has stable forking so does $T^{eq}$. We can ask the same question about weak stable Kim-forking. Here we adapt \cite[Proposition 2.2]{casanovas2018stable} and then discuss the reasons why we can not adapt the proof of Casanovas in its entirety. In this section we work with a theory $T$ with existence and weak stable Kim-forking.

\begin{prop} If $T$ has weak stable Kim-forking and existence, then $T^{eq}$ has weak stable Kim-forking over real parameters: Whenever $a_{F}\centernot\forkindep^{K}_{E}B$ for $E$ an algebraically closed set of real elements, $B\supseteq E$ possibly containing imaginaries and $a_{F}\in \mathbb{M}^{eq}$ there is a stable formula $\varphi(x,y,e)$ with $e\in E$ and $b\in B$ such that $a_{F}\models \varphi(x,b,e)$ and $\varphi(x,b,e)$ Kim-forks over $E$.
\end{prop}

\begin{proof}
Let $a_{F},E,B$ be as in the statement. We choose a saturated model $M\supseteq B$ such that $a_{F}\in dcl^{eq}(M)$ and a representative $a$ of $a_{F}$ such that $a\forkindep^{f}_{Ea_{F}}M$ (using existence). Then $a\centernot\forkindep^{K}_{E}M$. By assumption there is a stable formula $\varphi(x,y,e)$ for $e\in E$, and a tuple $m\in M$ such that $\models \varphi(a,m,e)$ and $\varphi(x,m,e)$ Kim-forks over $E$.

\vspace{10pt}
Let $p=\tp_{\varphi^{*}}(a/M)$. By \cref{trans} there is a stable formula $\psi(x,y,e)$ and $b\in B$ such that $\models \psi(a,b,e)$ and $\psi(x,b,e)$ Kim-forks over $E$. Since $a\forkindep^{K}_{Ea_{F}}B$ the $\psi(x,y,e)$-type of $a$ over $B$ is definable over $Ea_{F}$. Let $d\in acl^{eq}(Ea_{F})$ be its canonical base and $\chi(d,y)$ be its definition, by \cref{formstab} we can assume that $\chi(w,y)$ is a stable formula. 

\vspace{10pt}
\textbf{Claim:} If $q(w):=tp(d/E)$ then $q(w)\cup \lbrace \chi (w,b)\rbrace$ Kim-forks over $E$. 
\vspace{10pt}

\textit{Proof:} Consider an $E$-Morley sequence $(b_{i})_{i<\omega}$ with $b_{0}=b$. If $d'\models q(w)\cup \lbrace \chi (w,b_{i}) $ : $i<\omega\rbrace$ by saturation there is some $(b'_{i})_{i<\omega} \in M$ such that 
$d'(b_{i})_{i<\omega}\equiv_{E} d(b'_{i})_{i<\omega}$, so $\models \chi(d,b'_{i})$ for all $i<\omega$, so $\psi(x,b'_{i},e)\in \tp_{\psi}(a/M)$, so $\lbrace \psi(x,b'_{i},e)$ : $i<\omega \rangle$ is consistent, which contradicts the fact that $\psi(x,b,e)$ Kim-divides over $E$.\hspace*{0pt}\hfill\qedsymbol{}

\vspace{10pt}
There is some $\mu(w,e) \in q(w,e)$ such that $\chi'(w,b,e):=\chi(w,b)\wedge \mu(w,e)$ Kim-forks over $E$, also $\chi'(w,y,e)$ is stable. We know that $d\in acl^{eq}(Ea_{F})$, we can assume that $d\in acl^{eq}(ea_{F})$, take $\theta(w,a_{F},e)$ some formula witnessing this, i.e. $\models \theta(d,a_{F},e)$ and $\models \exists^{=n}w\theta(w,x,z)$.

\vspace{10pt}
Let $\alpha(v,y,e):=\exists w (\theta(v,w,e)\wedge \chi(w,y,e))$. By \cref{stablealg} $\alpha(v,y,e)$ is stable, $\models\alpha(a_{F},b,e)$ and $\alpha(v,b,e)$ Kim-forks over $E$. For the last point if there were $a'_{F}\models \alpha(v,b,e)$ such that $a'\forkindep^{K}_{E}b$, then there is $d'\in acl(ea')$ such that $\models \chi'(d',b,e)$, which contradicts $d'\forkindep^{E}b$.
\end{proof}

We would like to show that $T^{eq}$ satisfies weak stable forking. Consider $e_{F}$ a single imaginary, a tuple $a$ and a model $M$ containing $e_{F}$ such that $a\centernot\forkindep^{K}_{e}M$. We want to find some stable formula $\varphi(x,y,e_{F})$ and $m\in M$ such that $\models \varphi (a,m,e_{F})$ and $\varphi(x,m,e_{F})$ Kim-forks over $e_{F}$.

\vspace{10pt}
There is some representative $e$ of $e_{F}$ such that $a\forkindep^{f}_{e_{F}}Me$, so $a\centernot\forkindep_{e}M$. By assumption we can find a stable formula $\varphi(x,y,e)$ and $m\in M$ such that $\models \varphi (a,m,e)$ and $\varphi(x,m,e)$ Kim-forks over $e$. The problem now is that there is no obvious way of preserving these conditions when reducing $e$ to $e_{F}$. The same difficulty appears if we want to deduce weak stable Kim-forking from weak stable Kim-forking over models.

\section{Some questions}

We conclude this chapter with the following questions:

\begin{enumerate}
    \item Does existence and weak stable Kim-forking over models imply weak stable Kim-forking over arbitrary sets?
     \item If $T^{eq}$ has existence and weak stable Kim-forking, does $T^{eq}$ have weak stable Kim-forking?
    \item More specifically, can we prove that if a simple theory $T$ has weak stable forking then $T^{eq}$ also have weak stable forking?
    \item Is there an  theory that does not have weak stable Kim-forking? Either over models or algebraically closed sets if it has existence.
\end{enumerate}

\bibliographystyle{plain}
\bibliography{ref.bib}

\end{document}